\documentclass[11pt]{amsart}
\usepackage[english]{babel}
\usepackage{amssymb}
\usepackage{hyperref}
\usepackage{amsthm}

\numberwithin{equation}{section}       

\theoremstyle{plain}

\newtheorem{prop}{Proposition}[section]

\newtheorem{cor}[prop]{Corollary}
\newtheorem{lem}[prop]{Lemma}

\theoremstyle{definition}

\theoremstyle{remark}
\newtheorem{rem}[prop]{Remark}

\newtheoremstyle{citing}
  {3pt}
  {3pt}
  {\itshape}
  {}
  {\bfseries}
  {.}
  {.5em}
  {\thmnote{#3}}

\theoremstyle{citing}
\newtheorem*{generic}{}

\newcommand{\N}{\mathbb{N}}
\newcommand{\Z}{\mathbb{Z}}
\newcommand{\C}{\mathbb{C}}
\newcommand{\R}{\mathbb{R}}

\DeclareMathOperator{\Lip}{Lip}

\DeclareMathOperator{\Top}{top}
\DeclareMathOperator{\dist}{dist}
\DeclareMathOperator{\diam}{diam}

\newcommand{\CC}{\overline{\C}}
\newcommand{\tvarphi}{\widetilde{\varphi}}
\newcommand{\htop}{h_{\Top}}
\newcommand{\julia}{J(f)}
\renewcommand{\=}{ : = }
\newcommand{\partn}[1]{{\smallskip \noindent \textbf{#1.}}}

\newcommand{\hrho}{\widehat{\rho}}

\newcommand{\tphi}{\widetilde{\phi}}
\newcommand{\chiinf}{\chi_{\inf}}

\begin{document}
\title[A characterization of hyperbolic potentials]
{A characterization of hyperbolic potentials of rational maps}
\author[I. Inoquio-Renteria]{Irene Inoquio-Renteria$^\dag$}
\address{Irene Inoquio-Renteria, Depto. de Matem{\'a}tica, ICMC/USP - S{\~a}o Carlos, Caixa Postal 668, 13560-970  S{\~a}o Carlos, SP, Brazil}
\email{ireinoquio@icmc.usp.br}
\thanks{$\dag$ Partially supported by Partially supported by FAPESP 2010/07267-4}
\author[J. Rivera-Letelier]{Juan Rivera-Letelier$^\ddag$}
\address{Juan Rivera-Letelier, Facultad de Matem{\'a}ticas, Pontificia Universidad Cat{\'o}lica de Chile, Avenida Vicu{\~n}a Mackenna~4860, Santiago, Chile}
\email{riveraletelier@mat.puc.cl}
\thanks{$\ddag$ Partially supported by FONDECYT N 1100922}


\begin{abstract}
Consider a rational map~$f$ of degree at least~$2$ acting on its Julia set~$J(f)$, a H{\"o}lder continuous potential~$\varphi : J(f) \to \R$ and the pressure~$P(f, \varphi)$.
In the case where 
$$\sup_{\julia}\varphi < P(f,\varphi),$$
the uniqueness and stochastic properties of the corresponding equilibrium states have been extensively studied.
In this paper we characterize those potentials~$\varphi$ for which this property is satisfied for some iterate of~$f$, in terms of the expanding properties of the corresponding equilibrium states.
A direct consequence of this result is that for a non-uniformly hyperbolic rational map every H{\"o}lder continuous potential has a unique equilibrium state and that this measure is exponentially mixing.
\end{abstract}
\maketitle
\section{Introduction}
In their pioneer works, Sinai, Ruelle and Bowen gave a complete description of the thermodynamic formalism of a smooth uniformly hyperbolic diffeo\-morphism acting on a compact manifold and a H{\"o}lder continuous potential \cite{Sin72,Bow75,Rue76}.
In the one-dimensional setting there have been several extensions of these results to real or complex maps which are not necessarily uniformly hyperbolic.
The lack of uniform hyperbolicity is compensated by an extra hypothesis on the potential.
Notably, for a general complex rational map~$f$ of degree at least~$2$ acting on its Julia set~$J(f)$, there have been a wealth of results for a H{\"o}lder continuous potential~$\varphi : J(f) \to \R$ satisfying
\begin{quote}
\begin{gather*}
\label{e:immediate hyperbolicity}
(*) \qquad
\sup_{z \in J(f)} \varphi(z) < P(f, \varphi),
\end{gather*}
\end{quote}
where~$P(f, \varphi)$ denotes the pressure.
This includes a series of papers by Denker, Haydn, Przytycki and Urba{\'n}ski~\cite{DenUrb91e,DenPrzUrb96,Hay99,Prz90} on the uniqueness and the stochastic properties of the corresponding equilibrium states, extending previous results of Lyubich~\cite{Lju83} and Freire, Lopes and Ma\~n{\'e}~\cite{FreLopMan83,Man83}.
See also~\cite{BruKel98,BruTod08,IomTod10} for results for sufficiently regular interval maps, Baladi's book~\cite[\S$3.2$]{Bal00b} and references therein for piecewise monotone transformations and~\cite{UrbZdu0901,VarVia10} and references therein for (recent) results in higher dimensions.

In this paper we show that for a rational map~$f$ of degree at least~$2$, a H{\"o}lder continuous potential~$\varphi$ satisfies~$(*)$ for an iterate of~$f$ if and only if the Lyapunov exponent of each equilibrium state of~$f$ for the potential~$\varphi$ is positive.
A direct consequence of this result is that for a rational map satisfying the Topological Collet-Eckmann condition, every H{\"o}lder continuous potential satisfies~$(*)$ for an iterate of~$f$.
Combined with the work of Denker, Haydn, Przytycki and Urba{\'n}ski \emph{op.cit.}, this implies that for such~$f$ every H{\"o}lder continuous potential has a unique equilibrium state and that this measure is exponentially mixing.

To state our result more precisely, fix a rational map~$f$ of degree at least~$2$.
For an invariant probability measure~$\mu$ on~$J(f)$ we denote by~$h_{\mu}(f)$ its measure theoretic entropy and by
$$ \chi_\mu(f) \= \int \ln |f'|d\mu $$
its Lyapunov exponent.
We will say that~$\mu$ is \emph{exponentially mixing}, or that it has \emph{exponential decay of correlations}, if there are constants $C > 0$ and $\rho \in (0,1)$ such that for every integer $n\geq 1$, every bounded and measurable function $\phi : \julia \rightarrow \R$ and every Lipschitz continuous function $\psi : \julia \rightarrow \R$, we have
$$\left| \int (\phi \circ f^n) \cdot \psi d\mu - \int \phi d\mu \int \psi d\mu \right|
\le
C \| \phi \|_\infty \| \psi \|_{\Lip} \rho^n, $$
where $\| \phi \|_{\infty} = \sup_{z \in \julia}|\phi(z)|$ and $\| \psi \|_{\Lip} = \sup_{z,z' \in \julia, z \neq z'} \frac{| \psi(z)-\psi(z')|}{\vert z - z' \vert}.$

For a H{\"o}lder continuous potential~$\varphi : J(f) \to \R$, the pressure~$P(f, \varphi)$ is by definition,
$$ P(f, \varphi)
=
\sup \left\{ h_\mu(f) + \int \varphi d\mu : \mu \text{ invariant probability measure on~$J(f)$} \right\}. $$
A measure is an \emph{equilibrium state of~$f$ for the potential~$\varphi$} if it realizes the supremum above.

We will say that a H{\"o}lder continuous potential~$\varphi : J(f) \to \R$ is \emph{hyperbolic for~$f$} if for some integer~$n \ge 1$ the function~$S_n(\varphi) = \varphi + \varphi \circ f + \cdots + \varphi \circ f^{n - 1}$ satisfies
$$ \sup_{z \in \julia} S_n(\varphi)(z) < P(f, S_n(\varphi)). $$
\begin{generic}[Main Theorem]\label{main theorem}
For a rational map~$f$ of degree at least~$2$ and a H{\"o}lder continuous potential $\varphi: \julia \rightarrow \R$, the following conditions are equivalent:
\begin{enumerate}
\item[1.]
the potential~$\varphi$ is hyperbolic;
\item[2.]
the Lyapunov exponent of each equilibrium state of~$f$ for the potential~$\varphi$ is strictly positive.
\end{enumerate}
Furthermore, if these equivalent conditions are satisfied, then:
\begin{enumerate}    
\item[3.]
there is a unique equilibrium state~$\mu$ of~$f$ for the potential~$\varphi$.
Moreover the measure theoretic entropy of~$\mu$ is strictly positive and~$\mu$ is exponentially mixing.
\end{enumerate}
\end{generic} 
For a rational map~$f$ such that the Lyapunov exponent of every invariant probability measure on the Julia set is strictly positive, part~$2$ of the \hyperref[main theorem]{Main Theorem} is automatically satisfied for every H{\"o}lder continuous potential~$\varphi$.
So the following corollary is an immediate consequence of the \hyperref[main theorem]{Main Theorem}.
\begin{cor}\label{cor1}
Let~$f$ be a rational map of degree at least~$2$ such that the Lyapunov exponent of each invariant probability measure on~$J(f)$ is strictly positive.
Then for every H{\"o}lder continuous potential $\varphi : \julia \rightarrow \R$ properties~$1$, $2$ and~$3$ of the \hyperref[main theorem]{Main Theorem} hold.
\end{cor}
This result applies in particular to rational maps satisfying a strong form of non-uniform hyperbolicity, called the \emph{Topological Collet-Eckmann condition}.\footnote{However there are rational maps for which the Lyapunov exponent of each invariant probability measure on the Julia set is strictly positive and that do not satisfy the topological Collet-Eckmann condition, see~\cite[Corollary~$2$]{BruTod06}.}
It can be defined as the following strong form of Pesin's non-uniform hyperbolicity condition:
\begin{quote}    
There exists $\chi > 0$ such that the Lyapunov exponent of each invariant probability measure~$\mu$ on~$J(f)$ satisfies
$$ \chi_\mu(f) \ge \chi.$$
\end{quote}
See~\cite{PrzRivSmi03} for several equivalent formulations of this condition, including the original formulation in topological terms.
Although the topological Collet-Eckmann condition is very strong, the set of rational maps that satisfy it, but that are not uniformly hyperbolic, has positive Lebesgue measure in the space of rational maps of a given degree, see~\cite{Aspthesis,GraSwi00,Smi00} and references therein. 

We thus obtain the following corollary as a direct consequence of the previous one.
\begin{cor}\label{cor2}
Let~$f$ be a rational map of degree at least~$2$ satisfying the Topological Collet-Eckmann condition.
Then for every H{\"o}lder continuous potential $\varphi: \julia \rightarrow \R$ properties~$1$, $2$ and~$3$ of the Main Theorem hold.
\end{cor}
The existence and uniqueness of the equilibrium state where shown in~\cite[Theorem~A]{ComRiv11}

We now comment on the proof of the Main Theorem.
The implication~$1 \Rightarrow 2$ is obtained from the variational principle and Ruelle's inequality, through a combination of known arguments.
The implication~$1 \Rightarrow 3$ is easily deduced from the combined results of Denker, Haydn, Przytycki and Urba\'nski \emph{op.cit}.  
The main difficulty in the proof of the Main Theorem is in the proof of the implication~$2 \Rightarrow 1$, which we deduce from the following result.
\begin{generic}[Key Lemma]
\label{key lemma}
Let $f$ be a rational map of degree at least~$2$ and let~$\mu$ be an ergodic invariant probability measure on~$\julia$ whose Lyapunov exponent is strictly positive.
Then for every H{\"o}lder continuous potential~$\varphi : \julia \to \R$ and each~$t \ge 0$ we have
$$ P\left(f,\varphi-t\ln|f'|\right)
> \int \varphi d\mu - t \chi_\mu(f). $$
\end{generic}
For~$\varphi$ and~$t \ge 0$ as in this lemma, the pressure~$P(f, \varphi - t \ln |f'|)$ and the equilibrium states of~$f$ for the potential~$\varphi - t \ln |f'|$ are defined as before, through the variational principle.\footnote{Note however that when~$t > 0$ the usual topological pressure is equal to~$+ \infty$, whereas~$P(f, \varphi - t \ln |f'|) < + \infty$.}

Although we only need this result for~$t = 0$, we have included this more general formulation for future reference.
When~$\varphi$ is constant equal to~$0$ and~$t > 0$, this result appeared in the first and second version of the preprint~\cite{PrzRiv0806v2}, but it did not appear in the published version of this paper.
The proof given here is mostly a throughout revision of~\cite[Lemma~$8.2$]{PrzRiv0806v2}, although some additional arguments are needed to handle the case~$t = 0$ that we use here to prove the Main Theorem.

As in the preprint~\cite{PrzRiv0806v2}, we obtain the following result as a direct consequence of the \hyperref[key lemma]{Key Lemma}.
First note that if we put
$$ \chiinf(f)
\=
\inf \{ \chi_{\mu}(f) : \mu \text{ invariant probability measure on~$J(f)$} \}, $$
then for each~$t \ge 0$ we have~$P(f, - t \ln |f'|) \ge - t \chiinf$.
We call
$$ t_+ \= \sup \{ t \ge 0 : P(f, - t \ln |f'|) > - t \chiinf \} $$
the \emph{freezing point of~$f$}.
When the freezing point~$t_+$ is finite, it is strictly positive and by definition for each~$t \in [t_+, + \infty)$ we have~$P(f, - t \ln |f'|) = - t \chiinf$; so the function $t \mapsto P(f, - t \ln |f'|)$ cannot be real analytic at~$t = t_+$.
Although for a uniformly hyperbolic rational map~$f$ we always have~$t_+ = + \infty$, Makarov and Smirnov have shown that there are (generalized polynomial like) maps that satisfy the Topological Collet-Eckmann condition and whose freezing point is finite, see~\cite{MakSmi03}.
\begin{cor}
\label{c:after freezing}
Let~$f$ be a rational map satisfying the Topological Collet-Eckmann condition and whose freezing point~$t_+$ is finite.
Then the following properties hold.
\begin{enumerate}
\item[1.]
For each invariant probability measure~$\mu$ on~$J(f)$ we have~$\chi_{\mu}(f) > \chiinf(f)$.
\item[2.]
For each~$t \in (t_+, + \infty)$ there is no equilibrium state of~$f$ for the potential~$- t \ln |f'|$.
\item[3.]
There is at most one equilibrium state of~$f$ for the potential~$- t \ln |f'|$.
If such a measure exists, then its measure theoretic entropy is strictly positive and the pressure function $t \mapsto P(f, - t \ln |f'|)$ is not differentiable at~$t = t_+$.
\end{enumerate}
\end{cor}
The proof of this corollary is given at the end of~\S\ref{s:continuous hyperbolic}.

Note that part~$1$ of this result implies that for~$f$ as in the statement of this corollary the function~$\mu \mapsto \chi_{\mu}(f)$ is discontinuous.
This is not so surprising, as Bruin and Keller showed in~\cite[Proposition~$2.8$]{BruKel98} that this holds for every S-unimodal interval map satisfying the Collet-Eckmann condition.
 \subsection{Organization}\label{organiza_Paper}
After recalling a few preliminary facts in~\S\ref{preliminar}, in~\S\ref{s:continuous hyperbolic} we gather several properties of hyperbolic potentials, that hold for a general continuous map acting on a compact metric space.
After these general considerations, the \hyperref[main theorem]{Main Theorem} is easily deduced from the \hyperref[key lemma]{Key Lemma}.
The proof of Corollary~\ref{c:after freezing} is given at the end of~\S\ref{s:continuous hyperbolic}.

The rest of the paper is devoted to the proof of the \hyperref[key lemma]{Key Lemma}.
In the case where the measure is supported on an exceptional periodic orbit the proof is different; we treat this special case separately in~\S\ref{s:periodic case}.
The proof in the general case is based on the construction of a suitable ``Iterated Function System'' generated by the rational map (Proposition~\ref{p:construction of IFS} in~\S\ref{ss:IFSs}).
We deduce the \hyperref[key lemma]{Key Lemma} from this result in~\S\ref{proof of key lemma}, after some preparatory considerations in~\S\ref{Cantor repellers}.
The proof of Proposition~\ref{p:construction of IFS} is given in~\S\ref{s:construction of IFS}.

\subsection{Acknowledgements}
We would like to thank Feliks Przytycki and the referee for several comments and corrections.

\section{Preliminaries}\label{preliminar}
Throughout the rest of this paper we denote by~$\N$ the set of strictly positive integers.

We endow the Riemann sphere~$\CC$ with the spherical metric, that we denote by~$\dist$.
Distances, diameters, balls and derivatives are all taken with respect to the spherical metric.
For each $z \in \CC$ and $r > 0$ we denote by~$B(z,r)$ the open ball centered at~$z$ and of radius~$r$.

For basic references on the dynamics of rational maps, see~\cite{CarGam93,Mil06}.
The following is a version of the classical Koebe Distortion Theorem (see \emph{e.g.}~\cite{Pom92}), adapted to the spherical metric and to inverse branches of a given rational map.
\begin{generic}[Koebe Distortion Theorem]
For each rational map~$f$ of degree at least~$2$ there are constants~$\rho_0 > 0$ and~$K > 1$ such that the following property holds.
Let~$n \ge 1$ be an integer, $x$ a point in~$\CC$ and~$\rho \in (0, \rho_0)$, such that~$f^n$ maps a neighborhood of~$x$ univalently onto~$B(f^n(x), \rho)$.
Then for every pair of points~$z$ and~$z'$ in the connected component of~$f^{-n}(B(f^n(x), \rho/2))$ containing~$x$, we have
$$ K^{-1} \le |(f^n)'(z)| / |(f^n)'(z')| \le K. $$
\end{generic}
To deduce this result from the classical Koebe Distortion Theorem, fix a periodic orbit~$\mathcal{O}$ of period at least~$3$ of~$f$ and let~$\rho_0 > 0$ be sufficiently small such that for every point~$z_0$ in~$\CC$ the ball~$B(z_0, \rho_0)$ contains at most one point of~$\mathcal{O}$.
Thus for every~$\rho \in (0, \rho_0)$ and every integer~$n \ge 1$, each connected component~$W$ of~$f^{-n}(B(z_0, \rho))$ contains at most one point of~$\mathcal{O}$, so
$$ \diam(\CC \setminus W)
\ge 
\min \{ \diam(\mathcal{O} \setminus \{ z \}) : z \in \mathcal{O} \}
>
0. $$
After a change of coordinates that is an isometry with respect to the spherical metric we assume that~$\infty$ is in~$\mathcal{O}$, but not in~$W$.
Then~$W$ is contained in~$\C$ and the spherical metric on the connected component~$W'$ of~$f^{-n}(B(z_0, \rho/2))$ contained in~$W$ is comparable to the Euclidean metric up to a multiplicative factor that only depends on~$\min \{ \diam(\mathcal{O} \setminus \{ z \}) : z \in \mathcal{O} \}$.
Similarly, in a coordinate such that~$B(z_0, \rho) \subset \C$, the Euclidean metric on~$B(z_0, \rho/2)$ is comparable to the spherical one up to a multiplicative factor that only depends on~$\rho_0$.
So the statement above is a direct consequence of the classical Koebe Distortion Theorem.
\section{Hyperbolic potentials of a topological dynamical system}
\label{s:continuous hyperbolic}
In this section we gather several properties of hyperbolic potentials which hold for a general continuous map acting on a compact metric space; for background on the thermodynamic formalism in this setting, see for example~\cite{PrzUrb10,Rue78,Wal82}.
After these results we give the proof of the \hyperref[main theorem]{Main Theorem}, assuming the \hyperref[key lemma]{Key Lemma}.
The proof of Corollary~\ref{c:after freezing} is given at the end of this section.

Let~$X$ be a compact metric space and~$f: X\rightarrow X$ a continuous map.
We denote by $\mathcal{M}(X,f)$ the space of Borel probability probability measures on~$X$ that are invariant by~$f$.

As for rational maps, for a continuous potential~$\varphi : X \to \R$ the \emph{pressure} is defined by
$$ P(f, \varphi)
=
\sup_{\mu \in \mathcal{M}(X, f)} \left( h_\mu(f) + \int \varphi d\mu \right) $$
and a measure realizing this supremum is called an \emph{equilibrium state of~$f$ for the potential~$\varphi$}.
Furthermore, we say that a continuous potential~$\varphi : X \to \R$ is \emph{hyperbolic for~$f$}, if there exists an integer~$n \ge 1$ such that
\begin{equation}
\label{e:hyperbolicity}
\sup_{z \in X} S_n(\varphi)(z) < P(f^n, S_n(\varphi)).
\end{equation}
Note that there are hyperbolic potentials for which this conditions fails for~$n = 1$, see Remark~\ref{hyperbolic remark} below.
Since for each integer~$n \ge 1$ and each continuous potential~$\varphi : X \to \R$ we have~$P(f^n, S_n(\varphi)) = n P(f, \varphi)$, the potential~$\varphi$ is hyperbolic for~$f$ if and only if there is an integer~$n \ge 1$ such that
$$ \sup_X \frac{1}{n} S_n(\varphi) < P(f, \varphi). $$

In the following proposition we give several characterizations of hyperbolic potentials.
Two continuous potentials~$\varphi, \tvarphi : X \to \R$ are said to be \emph{co\nobreakdash-homologous} if there is a continuous function~$h : X \to \R$ such that
$$ \tvarphi = \varphi + h - h \circ f . $$
In this case, for each invariant probability measure~$\mu$ on~$X$ we have $\int \tvarphi d\mu = \int \varphi d\mu$; this implies that~$P(f, \tvarphi) = P(f, \varphi)$ and that the equilibrium states of~$f$ for the potentials~$\varphi$ and~$\tvarphi$, coincide.
\begin{prop}\label{p:continuous hyperbolicity}  
Let~$X$ be a compact metric space and let $f: X\rightarrow X$ be a continuous function. 
Then for every continuous potential $\varphi : X \rightarrow \R$ we have
\begin{multline}
\label{e:invariant supremum}
\lim_{n\rightarrow + \infty} \sup_{X}\frac{1}{n} S_n(\varphi) 
=
\inf \left( \sup_X \frac{1}{n} S_n(\varphi) \right)_{n = 1}^{+ \infty}
\\ =
\sup_{\mu\in\mathcal{M}(X,f)}\int \varphi d\mu \le P(f, \varphi).
\end{multline}
Furthermore, the following properties are equivalent.
\begin{enumerate}
\item[1.] The potential~$\varphi$ is hyperbolic for~$f$.
\item[2.]
$\sup_{\mu\in\mathcal{M}(X,f)}\int \varphi d\mu< P(f,\varphi).$
\item[3.]
The measure theoretic entropy of each equilibrium state of~$f$ for the potential~$\varphi$ is strictly positive.
\item[4.]
There exists a continuous potential~$\tvarphi$ co\nobreakdash-homologous to~$\varphi$ such that
$$ \sup_{X} \tvarphi< P(f,\tvarphi). $$
\item[5.]
Every continuous potential co\nobreakdash-homologous to~$\varphi$ is hyperbolic for~$f$.
     \end{enumerate}
    \end{prop}
We remark that in many situations the potential~$\tvarphi$ in property~$4$ can be taken to be as regular as~$\varphi$ and~$f$; in fact, in the proof we show that for a sufficiently large~$n$ we can take $\tvarphi = \frac{1}{n} S_n(\varphi)$.

The proof of this proposition is given after the following corollary.
For a map~$f$ as in the statement of the proposition, we denote by~$\htop(f)$ its topological entropy.
\begin{cor}
Let~$X$ be a compact metric space and let~$f : X \to X$ be a continuous map.
Then~$f$ has a hyperbolic potential if and only if~$\htop(f) > 0$.
\end{cor}
\begin{proof}
The direct implication follows from property~$3$ of Proposition~\ref{p:continuous hyperbolicity} and from the variational principle.
To prove the reverse implication, note that if~$\htop(f) > 0$ and if we denote by~$\varphi : X \to \R$ the constant potential equal to~$0$, then~$\varphi$ is hyperbolic for~$f$:
$$ P(f, \varphi)
=
\htop (f)
>
0
=
\sup_X \varphi. $$
\end{proof}
\begin{rem}
\label{hyperbolic remark}
When~$\htop(f) > 0$ there always exists a hyperbolic potential~$\varphi$ for which~\eqref{e:hyperbolicity} fails for~$n = 1$:
$$ \sup_X \varphi > P(f, \varphi). $$
To find such~$\varphi$, let~$x$ be a point in~$X$ which is not fixed by $f$ and let $h:~X\rightarrow \R$ be a continuous function such that $h(x)-~h(f(x)) > \htop(f)$.
Then the potential $\varphi \= h - h\circ f$ is co\nobreakdash-homologous to the constant potential equal to~$0$.
Thus~$\varphi$ is hyperbolic, $P(f, \varphi) = \htop(f)$ and we have,
$$ \sup_X \varphi 
\geq
h(x) - h(f(x))
>
\htop(f)
=
P(f,\varphi). $$
\end{rem}
\begin{proof}[Proof of Proposition~\ref{p:continuous hyperbolicity}]
The inequality in~\eqref{e:invariant supremum} is a direct consequence of the definition of~$P(f, \varphi)$.
To prove the equalities in~\eqref{e:invariant supremum}, note first that
for each integer~$n \ge 1$ and each measure $\mu\in~\mathcal{M}(X,f)$ we have
$$\int S_n(\varphi) d\mu= \sum_{k=0}^{n-1}\int \varphi\circ f^k d\mu = n\int \varphi d\mu,$$
so
$$\int \varphi d\mu\leq \sup_{X}\frac{1}{n} S_n(\varphi).$$
Taking the supremum over~$\mu$ in~$\mathcal{M}(X, f)$ and then the infimum over~$n$ in~$\{1, 2, \ldots \}$, we get
		      \begin{equation}\label{eq1}
		      \sup_{\mu\in\mathcal{M}(X,f)}\int \varphi d\mu
\leq
\inf \left( \sup_{X}\frac{1}{n}S_n(\varphi) \right)_{n = 1}^{+ \infty} .
\end{equation}
So, to complete the proof of~\eqref{e:invariant supremum} it is enough to show that
$$ \limsup_{n \to + \infty} \sup_X \frac{1}{n} S_n(\varphi)
\le
\sup_{\mu \in \mathcal{M}(X, f)} \int \varphi d\mu. $$
To do this, for each integer $n\geq 1$ let $x_n\in X$ be such that
$$\frac{1}{n}S_n(\varphi)(x_n)=\sup_{X}\frac{1}{n}S_n(\varphi) $$
and put
$$\mu_n\=\frac{1}{n}\sum_{j=0}^{n-1}\delta_{f^{j}(x_n)}.$$
		 Observe that $\mu_n(\varphi)=\frac{1}{n}S_n(\varphi)(x_n)= \sup_{X}\frac{1}{n}S_{n}\varphi$.
Choose a sequence~$(n_m)_{m = 1}^{+ \infty}$ of positive integers such that 
$$ \lim_{m\rightarrow + \infty}\frac{1}{n_m}S_{n_m}\varphi(x_{n_m})
=
\limsup_{n\rightarrow + \infty}\sup_{X} \frac{1}{n} S_n(\varphi).$$
Taking a subsequence if necessary we assume that the sequence $(\mu_{n_m})_{m = 1}^{+ \infty}$ converges in the weak* topology to a measure~$\mu$ in~$\mathcal{M}(X,f)$.
Then we have
\begin{multline*}\label{eq2}
\int \varphi d\mu
=
\lim_{m\rightarrow + \infty}\int \varphi d\mu_{n_m}
=
\lim_{m\rightarrow + \infty}\frac{1}{n_m} S_{n_m}\varphi(x_{n_m})
\\=
\limsup_{n\rightarrow + \infty }\sup_{X}\frac{1}{n}S_n(\varphi).
\end{multline*}
Together with~\eqref{eq1} this implies the equalities in~\eqref{e:invariant supremum} and completes the proof of~\eqref{e:invariant supremum}.

We will now prove the equivalence of properties~$1$--$5$.
The second equality in~\eqref{e:invariant supremum} easily implies that properties~$1$ and~$2$ are equivalent.
To prove the implication~$2 \Rightarrow 3$, let~$\varphi$ be a hyperbolic potential for~$f$ and let~$\mu$ be an equilibrium state of~$f$ for the potential~$\varphi$.
Then we have
$$ h_\mu(f)
=
P(f, \varphi) - \int \varphi d \mu
\ge
P(f, \varphi) - \sup_{\mu' \in \mathcal{M}(X, f)} \int \varphi d\mu'
> 0. $$
To prove implication~$3 \Rightarrow 2$ we proceed by contradiction and assume
$$ \sup_{\mu \in \mathcal{M}(X, f)} \int \varphi d\mu = P(f, \varphi). $$
This implies that there is an invariant probability measure~$\mu_0$ such that~$\int \varphi d\mu_0 = P(f, \varphi)$.
By definition of~$P(f, \varphi)$ we have,
$$ P(f, \varphi)
=
\int \varphi d\mu_0
\le
\int \varphi d \mu_0 + h_{\mu_0}(f)
\le
P(f, \varphi), $$
so~$\mu_0$ is an equilibrium state of~$f$ for the potential~$\varphi$ and $h_{\mu_0}(f) = 0$.
This completes the proof of the implication~$3 \Rightarrow 2$.

So far we have shown that properties~$1$, $2$ and~$3$ are equivalent.
To complete the proof of the proposition, first observe that, since for each continuous potential~$\tvarphi$ co\nobreakdash-homologous to~$\varphi$ the equilibrium states of~$f$ for the potentials~$\varphi$ and~$\tvarphi$ coincide, by property~$3$ the potential~$\tvarphi$ is hyperbolic if and only if~$\varphi$ is.
We obtain as a direct consequence of this fact the implication~$4 \Rightarrow 1$ and the equivalence between~$5$ and~$1$. 
Finally, observe that the implication $2 \Rightarrow 4$ is a direct consequence of the second equality in~\eqref{e:invariant supremum} and of the fact that for each integer $n\geq 1$ the function~$\frac{1}{n}S_n(\varphi)$ is co\nobreakdash-homologous to $\varphi$: if we put
$$ h=-\frac{1}{n}\sum_{j=0}^{n-2}(n-1-j)\varphi\circ f^j , $$
then 
$$ \frac{1}{n}S_n(\varphi)= \varphi+ h-h\circ f. $$
This completes the proof of the proposition.
\end{proof}

\begin{proof}[Proof of the Main Theorem assuming the Key~Lemma]
To prove the implication~$1 \Rightarrow 2$, let~$\varphi : J(f) \to \R$ be a hyperbolic potential and let~$\mu$ be an equilibrium state of~$f$ for the potential~$\varphi$.
Then by property~$3$ of Proposition~\ref{p:continuous hyperbolicity} we have~$h_\mu(f) > 0$, so by Ruelle's inequality we have
$$ \max \{ 2 \chi_\mu(f), 0 \} \ge h_{\mu}(f) > 0, $$
and thus~$\chi_\mu(f) > 0$; see for example~\cite{PrzUrb10,Rue78} for Ruelle's inequality.

To prove the reverse implication we proceed by contradiction: suppose there is a continuous and non\nobreakdash-hyperbolic potential $\varphi : J(f) \to \R$ such that the Lyapunov exponent of each of its equilibrium states is strictly positive.
Then by property~$3$ of Proposition~\ref{p:continuous hyperbolicity} there exists an equilibrium state of~$f$ for the potential~$\varphi$ such that~$h_{\mu_0}(f) = 0$.
Therefore~$\int \varphi d\mu_0 = P(f, \varphi)$ and by hypothesis~$\chi_{\mu_0}(f) > 0$.
Replacing~$\mu_0$ by one of its ergodic components if necessary, we assume~$\mu_0$ ergodic.
Then we can apply the~\hyperref[key lemma]{Key~Lemma} with $t=0$ and~$\mu = \mu_0$ and obtain that $\int\varphi d\mu_0 < P(f,\varphi)$.
This contradiction completes the proof of the implication~$2 \Rightarrow 1$.

Finally, we prove~$1 \Rightarrow 3$.
Let~$\varphi$ be a hyperbolic potential, so there is an integer $n \geq 1$ such that the H{\"o}lder continuous potential~$\tvarphi \= \frac{1}{n}S_{n}(\varphi)$ satisfies $\sup_{\julia} \tvarphi < P(f, \varphi)$.
Since~$\tvarphi$ is co\nobreakdash-homologous to~$\varphi$ (\emph{cf.} the proof of Proposition~\ref{p:continuous hyperbolicity}), we have~$P(f, \tvarphi) = P(f, \varphi)$ and the equilibrium states of~$f$ for the potentials~$\tvarphi$ and~$\varphi$ coincide.
Thus~$\sup_{\julia} \tvarphi < P(f, \tvarphi)$ and the combination of the works of Denker, Przytycki and Urba\'nski in~\cite{DenUrb91e,DenPrzUrb96} and of Haydn~\cite{Hay99} imply property~$3$ for the potential~$\tvarphi$, and hence for the potential~$\varphi$.
\end{proof}
\begin{proof}[Proof of Corollary~\ref{c:after freezing}]
Note that by the definition of~$t_+$, for each~$t \in [t_+, + \infty)$ we have~$P(f, - t \ln |f'|) = - t \chiinf$.

To prove part~$1$, observe that the \hyperref[key lemma]{Key Lemma} with~$\varphi$ equal to the constant function equal to~$0$ and~$t = t_+$, implies that for each invariant probability measure on~$J(f)$,
$$ - t_+ \chiinf
=
P(f, - t_+ \ln |f'|)
>
- t_+ \chi_{\mu}(f). $$
Thus $\chi_{\mu}(f) > \chiinf(f)$, as wanted.

To prove part~$2$, let~$t \in (t_+, + \infty)$ and let~$\mu$ be an invariant probability measure on~$J(f)$.
Then by part~$1$,
\begin{multline*}
h_{\mu}(f) - t \chi_{\mu}(f)
=
h_{\mu}(f) - t_+ \chi_{\mu}(f) - (t - t_+) \chi_{\mu}(f)
\\ \le
P(f, - t_+ \ln |f'|) - (t - t_+) \chi_{\mu}(f)
=
- t_+ \chiinf(f) - (t - t_+) \chi_{\mu}(f)
\\ <
- t \chiinf(t)
\le
P(f, - t \ln |f'|).
\end{multline*}
So~$\mu$ is not an equilibrium state of~$f$ for the potential~$-t \ln |f'|$.

It remains to prove part~$3$.
Since by definition of the Topological Collet-Eckmann condition the Lyapunov exponent of each invariant probability measure on~$J(f)$ is strictly positive, by \cite[Corollary~$11$]{Dob0804} there is at most one equilibrium state of~$f$ for the potential~$- t_+ \ln |f′|$, see also~\cite{Led84}.
If such a measure~$\mu$ exists, then we have
$$ h_{\mu}(f)
=
P(f, - t_+ \ln |f'|) + t_+ \chi_{\mu}(f)
=
t_+ (\chi_{\mu}(f) - \chiinf(f)) > 0. $$
On the other hand, since for each~$t$ we have
$$ P(f, - t \ln |f'|) \ge h_{\mu} - t \chi_{\mu} = P(f, - t_+ \ln |f'|) - (t - t_+) \chi_\mu(f), $$
we obtain
\begin{multline*}
\limsup_{t \to (t_+)^-} \frac{P(f, -t \ln |f'|) - P(f, -t_+ \ln |f'|)}{t - t_+}
\le
- \chi_{\mu}(f)
\\ <
\chiinf(f)
=
\lim_{t \to (t_+)^+} \frac{\partial}{\partial t} P(f, - t \ln |f'|).
\end{multline*}
\end{proof}

\section{The periodic case}
\label{s:periodic case}
This section is devoted to the special case of the \hyperref[key lemma]{Key Lemma}, where the measure~$\mu$ is supported on a periodic orbit; we state it as Proposition~\ref{p:periodic case} below.
The proof in the non\nobreakdash-periodic case is based on some of the same ideas as in the periodic case, but it is more elaborated.
\begin{prop}
\label{p:periodic case}
Let~$f$ be a rational map of degree at least~$2$, $\varphi : \julia \to \R$ a H{\"o}lder continuous potential and let $t \ge 0$.
Then for each integer~$n \ge 1$ and each repelling periodic point~$z_0$ of~$f$ of period~$n$ we have
\begin{equation}
P(f,\varphi - t \ln |f'|)
>
\frac{1}{n} S_n(\varphi)(z_0) - t \ln |(f^n)'(z_0)|.
\end{equation}
\end{prop}
\begin{proof}
Replacing~$f$ by an iterate if necessary we assume~$z_0$ is a fixed point of $f$.
Let $\rho>0$ be sufficiently small so that there is a local inverse ${\phi}: B(z_0,\rho) \rightarrow \CC$ of~$f$ fixing~$z_0$.
Reducing~$\rho$ if necessary we assume 
$$ \phi \left(\overline{B(z_0,\rho)}\right) \subseteq B(z_0,\rho) $$
and that there is~$\theta \in (0, 1)$ such that~$\phi$ contracts distances at least by a factor of~$\theta$.
Since~$z_0$ is a repelling fixed point of~$f$ there is a point~$z_1$ in $B(z_0,\rho/2)$ different from~$z_0$ and an integer $N \geq 1$ such that $f^N(z_1)=z_0$.
Taking~$N$ larger if necessary we assume that the connected component~$U_1$ of $f^{-N}(B(z_0,\rho))$ containing~$z_1$ is contained in $B(z_0, \rho/2)$ and that its closure is disjoint from that of $U_0\= {\phi}^N(B(z_0,\rho))$.
Taking~$N$ even larger if necessary, we also assume that~$f^N$ has no critical points in $U_1\backslash \{z_1\}$.
    
Put
$$ U = U_0 \cup U_1
\text{, }
\widehat{f}\= f^N\vert_{U}:U\rightarrow B(z_0,\rho)
\text{ and }
\widehat{\varphi}= \frac{1}{N} S_N(\varphi). $$
Note that~$\widehat{f}$ maps each of the sets~$U_0$ and~$U_1$ properly onto $B(z_0,\rho)$ and that~$\widehat{f}$ has no critical points, except maybe for $z = z_1$.
For an integer $k\geq 1$, a function $\psi:\julia\rightarrow\R$ and a point~$x$ in the domain of definition of~$\widehat{f}^k$ we put
$$ \widehat{S}_k(\psi)(x)
=
\psi(x) + \psi\circ \widehat{f} (x) + \cdots + \psi \circ \widehat{f}^{k-1}(x).$$
We also put
$$ \widehat{K}
\=
\{x\in U:\textrm{ for every integer } k\geq 1 \textrm{ the map } \widehat{f}^k\textrm{ is defined at } x\}.$$
Then we have  
	          \begin{multline*}
	          P\left(f,\varphi-t\ln \vert f'\vert\right)=\frac{1}{N} P\left(f^N, S_N(\varphi)-t\ln \vert (f^N)'\vert\right)
\\ \geq
\frac{1}{N} P\left(\widehat{f}\vert_{\widehat{K}}, \widehat{\varphi}-t\ln            
              \left\vert\widehat{f}'\right\vert\right).
       	      \end{multline*} 
    So to prove the lemma we just need to prove that, 
    $$ P\left(\widehat{f}\vert_{\widehat{K}}, \widehat{\varphi}-t\ln \vert \widehat{f}' \vert\right)> \widehat{\varphi}(z_0)- 
    t\ln \left\vert \widehat{f}'(z_0)\right\vert.$$

To do this, consider the itinerary map
$$ \begin{array}{rcl}
\iota :  \widehat{K} & \rightarrow & \{0,1\}^{\N\cup\{0\}} \\
z & \mapsto & \iota(z)_0\iota(z)_1\cdots
\end{array} $$
in such a way that for every $j \in \N \cup \{ 0 \}$ we have $\widehat{f}^{j}(z)\in U_{{\iota(z)}_j}$.
For each integer $k\geq 0$ and each sequence $a_0, \ldots, a_k\in \{0,1\}$ we put 
$$ K(a_0 \cdots a_k)
\=
\left\{z\in\widehat{K}: \text{ for every } j \in \{ 0, \ldots, k \} \text{ we have } \iota(z)_{j}=a_j \right\}. $$
By our choice of~$\phi$ it follows that there is a constant~$\widehat{S} > 0$ such that for each integer~$k \ge 1$ and each point~$x$ in~$K(\underbrace{0 \cdots 0}_{k})$ we have
$$ |\widehat{S}_k(\widehat{\varphi})(x) - k \widehat{\varphi}(z_0)| < \widehat{S}. $$
Taking~$\widehat{S}$ larger if necessary we assume that for every point~$x$ in~$U$ we have
$$ |\widehat{\varphi}(x) - \widehat{\varphi}(z_0)| < \widehat{S}. $$

We have two cases.
      
\partn{Case~$1$} $t=0$.
Since~$\widehat{f}$ has no critical points, except maybe for~$z_1$, and since $\widehat{f}(z_1)= z_0$ is fixed by $\widehat{f}$, it follows that $\widehat{f}$ is semi-hyperbolic in the sense of~\cite{CarJonYoc94}.
Thus for each integer $k \geq 1$ and each sequence $a_0, \ldots, a_k \in \{0,1\}$ the diameter of each connected component of $K(a_0 \cdots a_k)$ is exponentially small in~$k$.
This implies that for each $x_0\in \widehat{K}$ we have 
$$ P\left(\widehat{f}\vert_{\widehat{K}},\widehat{\varphi}\right)
=
\lim_{n\rightarrow +\infty}\frac{1}{n}\ln \sum_{x\in \widehat{f}^{-n}(x_0)}\exp\left(\widehat{S}_n(\widehat{\varphi})(x)\right),$$
see for example~\cite[Lemma~$4.2$]{ComRiv11}.
Fix a point~$x_0$ in~$\widehat{K}$.
Then the radius of convergence~$R$ of the power series in the variable~$s$,
               \begin{equation}
               \Xi(s) \= \sum_{n=0}^{+ \infty} \left(
               \sum_{x\in \widehat{f}^{-n}(x_0)} 
               \exp\left(\widehat{S}_n(\widehat{\varphi})(x)\right)\right) s^n.
               \end{equation}
is equal to
$$R
=
\exp\left(-P\left(\widehat{f}\vert_{\widehat{K}}, \widehat{\varphi}\right)\right).$$
Thus we just need to show that~$R$ is strictly smaller than $\exp(-\widehat{\varphi}(z_0))$.

To do this, note first that, since~$\widehat{f}$ is proper, for each integer~$k \ge 0$ and each sequence~$a_0, \ldots, a_k \in \{0, 1 \}$ there is a point of~$\widehat{f}^{-k}(x_0)$ in~$K(a_0\ldots a_k)$.
We will prove now that for each integer $k \geq 1$, each sequence $a_0, \ldots, a_k \in \{0,1\}$ with~$a_0 = 1$ and each point~$x$ in~$K(a_0 \cdots a_k)$ we have
\begin{equation}
\label{e:distortion in case 1}
\widehat{S}_{k+1}(\widehat{\varphi})(x)
\geq 
(k+1)\widehat{\varphi}(z_0) - 2 (a_0 + \cdots + a_k) \widehat{S}.
\end{equation}
To prove this inequality, put~$\ell \= a_0 + \cdots + a_k$ and let~$i_1 = 0 < i_2 < \cdots < i_{\ell} \le k$ be all integers~$i$ in~$\{0, \ldots, k \}$ such that~$a_{i} = 1$.
We also put~$i_{\ell + 1} = k + 1$.
Then by our choice of~$\widehat{S}$, for each~$j \in \{ 1, \ldots, \ell \}$ we have
$$ \widehat{S}_{i_{j + 1} - i_j}(\widehat{f}^{i_j}(x)) \ge (i_{j + 1} - i_j) \widehat{\varphi}(z_0) - 2\widehat{S}. $$
Summing over~$j$ in~$\{ 1, \ldots, \ell \}$ we obtain~\eqref{e:distortion in case 1}.

To prove that~$R$ is strictly smaller than $\exp(-\widehat{\varphi}(z_0))$, we define the following power series in the variable~$s$,
$$ \Phi(s)
\=
\sum_{k = 1}^{+ \infty} \exp\left(k\widehat{\varphi}(z_0)-2\widehat{S}\right)s^k.$$
In view of~\eqref{e:distortion in case 1}, each of the coefficients of the series
$$ \Phi(s) + \Phi(s)^2 + \Phi(s)^3 + \cdots $$
is less than or equal to the corresponding coefficient of~$\Xi(s)$. 
On the other hand the radius of convergence of~$\Phi$ is equal to $\exp(-\widehat{\varphi}(z_0))$, so 
$$ \lim_{s\rightarrow \exp(-\widehat{\varphi}(z_0))^-} \Phi(s)= + \infty $$
and hence that there is $s_0 \in (0, \exp(-\widehat{\varphi}(z_0)))$ such that $\Phi(s_0)\geq 1.$
      It follows that the radius of convergence~$R$ of~$\Xi$ is less than or equal to~$s_0$.
We thus have,
$$ \exp \left( - P\left(\widehat{f}\vert_{\widehat{K}},\widehat{\varphi}\right) \right)
\le
R
\le
s_0
<
\exp( - \widehat{\varphi}(z_0)) $$
and hence $P\left(\widehat{f}\vert_{\widehat{K}},\widehat{\varphi}\right) > \widehat{\varphi}(z_0)$, as wanted.

\partn{Case~$2$} $t>0$. If $z_1$ is not a critical point of $\widehat{f}$ then $\ln\vert\widehat{f}'\vert$ is 
     bounded and H{\"o}lder continuous  on $U$, so this case follows from Case~$1$ applied to the H{\"o}lder 
     continuous potential $\widehat{\varphi}-t\ln\vert \widehat{f}'\vert$ instead of $\widehat{\varphi}$.
     So we assume that $z_1$ is a critical point of $\widehat{f}$.
Denote by~$d \geq 2$ the local degree
     of $\widehat{f}$ at $z_1$.
Let~$p_1$ be a fixed point of~$\widehat{f}$ in~$U_1$ and for each integer $k \geq 2$ let~$p_k$ be a fixed point of~$\widehat{f}^k$ contained in $K(\underbrace{0\cdots 0}_{k - 1}1)$.
By definition of~$\widehat{S}$, for each~$k \ge 1$ we have
$$ \widehat{S}_k(\widehat{\varphi})(p_k)
\geq
k \widehat{\varphi}(z_0) - 2\widehat{S}. $$
In particular we have
$$ \liminf_{k \rightarrow + \infty} \frac{1}{k} S_k(\widehat{\varphi})(p_k)
\ge
\widehat{\varphi}(z_0). $$
On the other hand a direct computation shows
$$ \lim_{k\rightarrow + \infty}\frac{1}{k} \ln \left\vert\left(\widehat{f}^k\right)'(p_k)\right\vert
=
\frac{1}{d} \ln \left\vert \widehat{f}'(z_0) \right\vert. $$
It follows that for a sufficiently large~$k$
             \begin{multline*}
P\left(\widehat{f}\vert_{\widehat{K}},\widehat{\varphi}-t\ln\left\vert \widehat{f}'\right\vert\right)\\ 
\geq
\frac{1}{k} \widehat{S}_k(\widehat{\varphi})(p_k) - t \frac{1}{k} \ln \left\vert\left(\widehat{f}^k \right)'(p_k)\right\vert\\
              > \widehat{\varphi}(z_0)-t\ln\vert \widehat{f}'(z_0)\vert.
             \end{multline*}
\end{proof}

\section{Iterated Function Systems}
\label{s:IFSs}
In this section we prove the \hyperref[key lemma]{Key~Lemma} assuming the existence of a suitable ``Iterated Function System'' generated by the rational map.
We state this fact as Proposition~\ref{p:construction of IFS}, below.
Its proof is postponed to~\S\ref{s:construction of IFS}.

After fixing some notations in~\S\ref{symbolic space}, we give the statement of Proposition~\ref{p:construction of IFS} in~\S\ref{ss:IFSs}.
The proof of the~\hyperref[key lemma]{Key~Lemma} is given in~\S\ref{proof of key lemma}, after some preparatory considerations in~\S\ref{Cantor repellers}.  
\subsection{Symbolic space}
\label{symbolic space}
Let~$\Sigma \= \{1,2,\cdots\}^{\N}$ be the space of all infinite words in the alphabet~$\N$ and for each $n\in\N$ define
$$ \Sigma_n\=\{1,2,\cdots \}^n $$
and put
$$ \Sigma^*\= \bigcup_{n\geq 1}\Sigma_n. $$
Given an integer~$n \ge 1$ and a sequence $\ell_1\cdots\ell_n$ in~$\Sigma_n$ we call~$n$ \emph{the length of $\ell_1 \cdots \ell_n$} and denote it by $\vert \ell_1 \cdots \ell_n \vert$. 
\subsection{Iterated Function Systems}
\label{ss:IFSs}
  An infinite sequence of pairwise distinct holomorphic maps $(\phi_\ell)_{\ell = 1}^{+ \infty}$ is called an \emph{Iterated Function System (IFS)}, if there are $z_0\in\CC$ and $\rho>0 $ such that for each $\ell\geq 1$ the map $\phi_\ell$ is defined on $B(z_0,\rho)$ and takes images in $B(z_0,\rho/2)$.
Then for each integer~$n \ge 1$ and word $\ell_1 \cdots \ell_n \in \Sigma_n$ we put
$$ \phi_{\ell_1 \cdots \ell_n}\=\phi_{\ell_1}\circ\cdots\circ \phi_{\ell_n}. $$
We say such an IFS is \emph{free} if for every pair of distinct finite words~$\underline{\ell}, \underline{\ell}' \in \Sigma^*$ the maps~$\phi_{\underline{\ell}}$ and~$\phi_{\underline{\ell}'}$ are distinct.

Given a rational map $f$ of degree at least~$2$, we say that such 
  an IFS is \emph{generated} by $f$, if there is a sequence of positive integers $(m_{\ell})_{\ell = 1}^{+ \infty}$ such that for each~$\ell$ the map~$f^{m_{\ell}}\circ\phi_{\ell}$ is the identity on $B(z_0,\rho)$~; in this case~$\phi_{\ell}$ is the inverse branch of~$f^{m_\ell}$ defined on~$B(z_0, \rho)$ and maps~$z_0$ to~$\phi_{\ell}(z_0)$.
We call~$(m_{\ell})_{n = 1}^{+ \infty}$ the \emph{time sequence of $(\phi_{\ell})_{\ell = 1}^{+ \infty}$} and for each integer~$n \ge 1$ and each finite word~$\ell_1 \cdots \ell_n \in \Sigma_n$ we put
$$ m_{\ell_1 \cdots \ell_n}\= m_{\ell_1}+\cdots +m_{\ell_n}. $$      
Furthermore, we say that $(\phi_\ell)_{\ell = 1}^{+ \infty}$ is \emph{hyperbolic with respect to~$f$} if there are constants $C>0$ and $\lambda >1$ such that for all $z\in B(z_0,\rho)$, $\underline{\ell}\in \Sigma^*$ and~$j\in\{1,\cdots, m_{\underline{\ell}}\}$ we have 
$$\left\vert( f^{j})'\left(f^{m_{\underline{\ell}}-j}(\phi_{\underline{\ell}}(z))\right) \right\vert
\geq
C\lambda^j.$$

Given a point $z\in~\CC$ and an integer $n \ge 1$ we will say that a point $y\in f^{-n}(z)$ is an \emph{unramified pre-image of~$z$ by~$f^n$} if $(f^{n})'(y)\neq 0.$
We will say~$z$ is \emph{exceptional} if it has at most finitely many unramified pre-images.
An exceptional point can have at most~$4$ unramified preimages, see~\cite[Proposition~$2.2$]{McM00} and also~\cite{MakSmi00}.
\begin{prop}
\label{p:construction of IFS}
Let~$f$ be a rational map of degree at least~$2$, $\rho_0 > 0$ be given by the Koebe Distortion Theorem, $\varphi : \julia \to \R$ a H{\"o}lder continuous potential and $t\geq 0$ and put~$\psi \= \varphi - t \ln |f'|$.
Let~$\mu$ be an ergodic invariant probability measure not charging an exceptional point and whose Lyapunov exponent is strictly positive.
Then there are~$C > 0$, ~$z_0 \in \CC$, $\rho \in (0, \rho_0)$ and a free IFS $(\phi_\ell)_{\ell = 1}^{+ \infty}$ generated by~$f$, which is defined on $B(z_0,\rho)$, is hyperbolic with respect to~$f$ and such that, if we denote by~$(m_{\ell})_{\ell = 1}^{+ \infty}$ its time sequence, then for every integer $\ell\geq 1$ we have
\begin{equation}\label{e:key estimate}
S_{m_{\ell}}(\psi)(\phi_{\ell}(z_0))
\ge
m_{\ell}\int \psi d\mu - C.
\end{equation}
\end{prop}
We will use the following well-known distortion lemma, see for example
~\cite[Lemma~$3.4.2$ \& Lemma~$5.2.2$]{PrzUrb10}.
The proof uses the H{\"o}lder continuity of the potential~$\varphi$ and the Koebe Distortion Theorem.
\begin{lem}[Bounded distortion of a hyperbolic IFS]
\label{lem1}
Let~$f$ be a rational map of degree at least~$2$, $\rho_0 > 0$ given by Koebe Distortion Theorem, $\varphi : \julia \to \R$ a H{\"o}lder continuous potential and~$t \ge 0$ and put~$\psi = \varphi - t \ln |f'|$.
Furthermore, let $z_0 \in \CC$, $\rho \in (0, \rho_0)$, let $(\phi_\ell)_{\ell = 1}^{+ \infty}$ be a IFS generated by~$f$ defined on~$B(z_0,\rho)$ and let~$(m_{\ell})_{\ell = 1}^{+ \infty}$ be its time sequence.
If $(\phi_\ell)_{\ell = 1}^{+ \infty}$ is hyperbolic with respect to~$f$, then there exists $C_0 > 0$ such that for each $\underline{\ell} \in \Sigma^*$ and $\xi, \xi' \in B(z_0,\rho/2)$
		    \begin{equation}
		     \left\vert S_{m_{\underline{\ell}}}(\psi)\left(\phi_{\underline{\ell}}(\xi)\right)-S_{m_{\underline{\ell}}}
		     (\psi)\left(\phi_{\underline{\ell}}(\xi')\right)\right\vert\leq C_0.
		    \end{equation}
	    \end{lem}
\subsection{Uniformly expanding sets and topological pressure}
\label{Cantor repellers}
Let $f$ be a rational map of degree at least~$2$, $z_0 \in \CC$, $\rho > 0$ and let $(\phi_\ell)_{\ell = 1}^{+ \infty}$ be an IFS generated by~$f$ defined on $B(z_0,\rho)$ with time sequence~$(m_{\ell})_{\ell = 1}^{+ \infty}$.
Then for each
$N\geq 1$ we put
$$ \widetilde{\Sigma}_N
\=
\{\underline{\ell}\in \Sigma^*:m_{\underline{\ell}}=N\} $$
 and
     		\begin{equation}\label{eq13}
	     	U_N=\bigcup_{\underline{\ell}\in\widetilde{\Sigma}_N} \phi_{\underline{\ell}}(B(z_0,\rho)).
     		\end{equation}
We also define $F \= f^N : U_N\rightarrow B(z_0,\rho)$, so that for each $\underline{\ell}\in \widetilde{\Sigma}_N$ the map~$F\circ\phi_{\underline{\ell}}$ is the identity on~$B(z_0,\rho)$.
Note that the set 	
$$ \widetilde{K}_N
\=
\bigcap_{j=1}^{+ \infty} \bigcup_{\underline{\ell}^1,\cdots, \underline{\ell}^j\in\widetilde{\Sigma}_N} 
  \phi_{\underline{\ell}^1}\circ\cdots \circ\phi_{\underline{\ell}^j}(B(z_0,\rho)),$$
is a Cantor set contained in~$\julia$ and that $F|_{\widetilde{K}_N}$ is uniformly expanding.
When $(\phi_\ell)_{\ell = 1}^{+ \infty}$ is free, then for each pair of distinct words~$\underline{\ell}$ and~$\underline{\ell}'$ in~$\widetilde{\Sigma}_N$ the maps~$\phi_{\underline{\ell}}$ and~$\phi_{\underline{\ell}'}$ are distinct inverse branches of~$F$ defined on~$B(z_0, \rho)$.
In particular the sets~$\phi_{\underline{\ell}}(B(z_0, \rho))$ and~$\phi_{\underline{\ell}'}(B(z_0, \rho))$ are disjoint.
It follows that for each continuous potential $\varphi:\widetilde{K}_N\rightarrow\R$ the topological pressure of~$F|_{\widetilde{K}_N}$ for the potential $S_N(\varphi)$ is given for each $\zeta\in B(z_0,\rho)$ by
             	\begin{multline*}
	     	      P\left(F, S_N(\varphi)\right)
\\=
\limsup_{n\rightarrow + \infty}\frac{1}{n}\ln\sum_{\underline{\ell}^1,\cdots,
		          \underline{\ell}^n\in\widetilde{\Sigma}_{N}}\exp\left( S_{nN}(\varphi)\left(\phi_{\underline{\ell}^{1}}
		          \circ\cdots\circ\phi_{\underline{\ell}^{n}}(\zeta)\right)\right).
		         \end{multline*}
           \begin{lem}\label{lem2}
Let~$f$ be a rational map of degree at least~$2$, $\varphi : \julia \to \R$ a H{\"o}lder continuous potential and~$t \ge 0$.
Furthermore, let $z_0 \in \CC$, $\rho > 0$ and let $(\phi_{\ell})_{\ell = 1}^{+ \infty}$ be a free IFS generated by~$f$ defined on $B(z_0,\rho)$, which is hyperbolic with respect to~$f$.
Fix~$\zeta$ in $B(z_0,\rho/2)$ and for each integer~$N \ge 1$ let $\widetilde{\Sigma}_N$ and $\widetilde{K}_N$ be as above and put
$$ \widetilde{\Lambda}_N
\=
\sum_{\underline{\ell}\in\widetilde{\Sigma}_N}\exp\Big(S_N(\varphi)
		    \left(\phi_{\underline{\ell}}(\zeta)\right)\Big)\left|\left(f^N\right)'\left(\phi_{\underline{\ell}}(\zeta)\right)\right|^{-t}.$$
		    Then we have
$$P\left(f,\varphi-t\ln|f'|\right)\geq\limsup_{N\rightarrow + \infty}\frac{1}{N}\ln\widetilde{\Lambda}_N.$$
            \end{lem}
\begin{proof}
Put $\psi=\varphi-t\ln|f'|$.
If we let $C_0 > 0$ be the constant given by Lemma~\ref{lem1}, then for each integer~$N \ge 1$ we have
\begin{multline*}
P(f, \psi)
=
\frac{1}{N} P(f^N, S_N(\psi))
\ge
\frac{1}{N}P\left(f^{N}|_{\widetilde{K}_N},S_N(\psi)\right)
\\ =
\frac{1}{N}\limsup_{n\rightarrow + \infty}\frac{1}{n}\ln\sum_{\underline{\ell}^1,\cdots,\underline{\ell}^n\in \widetilde{\Sigma}_N}
\exp \left[
S_{N}(\psi)\left(\phi_{\underline{\ell}^1}\circ\cdots\circ \phi_{\underline{\ell}^n} (\zeta) \right)
\right. \\ \left.
         +S_N(\psi)\left(\phi_{\underline{\ell}^2}\circ\cdots \circ\phi_{\underline{\ell}^n}(\zeta)\right)+\cdots+ S_N(\psi)
	        \left(\phi_{\underline{\ell}^n}(\zeta)\right)\right]
\\ \geq
\frac{1}{N} \limsup_{n\rightarrow + \infty} \frac{1}{n} \ln
\sum_{\underline{\ell}\in \widetilde{\Sigma}_N}\exp \left[ n \left( S_N(\psi)(\phi_{\underline{\ell}}(\zeta)) - C_0 \right) \right]
\\ =
\frac{-C_0}{N}+\frac{1}{N}\ln\widetilde{\Lambda}_N.
		\end{multline*}
We obtain the desired inequality by taking~$N \to + \infty$.
		      \end{proof}

\subsection{Proof of the~\hyperref[key lemma]{Key~Lemma} assuming Proposition~\ref{p:construction of IFS}}
\label{proof of key lemma}
Let~$\mu$ be an ergodic invariant probability measure on~$\julia$ whose Lyapunov exponent is strictly positive, $\varphi : \julia \to \R$ a H{\"o}lder continuous potential and $t \ge 0$ and put~$\psi \= \varphi - \ln |f'|$.
In the case where~$\mu$ is supported on a periodic orbit of~$f$, the desired inequality is given by Proposition~\ref{p:periodic case}.
Therefore from now on we will suppose that the measure~$\mu$ is not supported on a periodic orbit.
In particular, $\mu$ does not charge an exceptional point of~$f$ and hence it satisfies the hypothesis of Proposition~\ref{p:construction of IFS}.
Let $C > 0$, $z_0\in\julia$, $\rho>0$ and let $(\phi_\ell)_{\ell = 1}^{+ \infty}$ be the free IFS generated by~$f$ defined on $B(z_0,\rho)$, which is hyperbolic with respect to~$f$, given by Proposition~\ref{p:construction of IFS}.
Choose a non\nobreakdash-periodic point $\zeta\in~B(z_0,\rho/2)$ and for each integer~$N \ge 1$ let~$\widetilde{\Lambda}_N$ be as in Lemma~\ref{lem2}.
Then this lemma implies that the radius of convergence~$R$ of the power series in the variable~$s$ defined by,
$$\Xi(s)
\=
\sum_{N=1}^{+ \infty}\widetilde{\Lambda}_N s^N 
=
\sum_{\underline{\ell} \in \Sigma^*} \left( \exp\Big(S_{m_{\underline{\ell}}}(\psi) \left(\phi_{\underline{\ell}}(\zeta)\right)\Big) \right) s^{m_{\underline{\ell}}}$$
satisfies
$$ R
\=
\left( \limsup_{N\rightarrow + \infty}\widetilde{\Lambda}_N^{1/N} \right)^{-1}
\ge
\exp\left(-P(f,\psi)\right). $$
Therefore, to prove the~\hyperref[key lemma]{Key~Lemma} it is enough to prove that~$R$ is strictly less than $\exp(- \int \psi d\mu)$.

To prove this, let $C_0>0$ be given by Lemma~\ref{lem1} and let~$(m_{\ell})_{\ell = 1}^{+ \infty}$ be the time sequence of~$(\phi_{\ell})_{\ell = 1}^{+ \infty}$.
          By~\eqref{e:key estimate} and by Lemma~\ref{lem1}, for each integer $\ell\geq 1$ we have
          \begin{displaymath}
\exp(S_{m_\ell}(\psi)(\phi_{\ell}(\zeta)))
\ge
\exp(- (C_0 + C)) \exp\left(m_\ell\int \psi d\mu\right).
          \end{displaymath}   
So for each~$\underline{\ell} \in \Sigma^*$ we have,
 \begin{equation}\label{eq7}
                  \exp\left(S_{m_{\underline{\ell}}}(\psi)(\phi_{\underline{\ell}}(\zeta))\right)
\geq 
\exp(- \vert\underline{\ell}\vert (C_0 + C)) \exp\left(m_{\underline{\ell}}\int \psi d\mu\right).                              
\end{equation}
Therefore, if we define the power series~$\Phi$ in the variable~$s$ by,
\begin{displaymath}
\Phi(s)
\=
\sum_{\ell = 1}^{+ \infty} \exp\left(- C_0 - C + m_{\ell}\int\psi d\mu\right) s^{m_\ell},
\end{displaymath}                 
                then by~\eqref{eq7} each of the coefficients of the power series in the variable~$s$
                 \begin{equation}
                 \Phi(s)+\Phi(s)^2+\Phi(s)^3+\cdots \\
                 \end{equation}   
                 is less than or equal to the corresponding coefficient of the series~$\Xi$. 
                But the radius of convergence of~$\Phi$ is equal to $\exp\left(\displaystyle - \int\psi d\mu\right)$ and we thus have
                $$\lim_{s\rightarrow \exp\left(- \int\psi d\mu\right)^{-}}\Phi(s)=+\infty.$$
So there is $s_0 \in \left(0,\exp\left(- \int \psi d\mu\right)\right)$ such that $\Phi(s_0)\geq 1$ and therefore the radius of convergence~$R$ of~$\Xi$ is less than or equal to~$s_0$.
Hence we have
$$ \exp\left(-P(f,\psi)\right) \le R \le s_0 < \exp \left( - \int \psi d\mu \right) $$
and $P(f, \psi) > \int \psi d\mu$, as wanted.

\section{Constructing the IFS}
\label{s:construction of IFS}
In this section we prove Proposition~\ref{p:construction of IFS}, and thus complete the proof of the \hyperref[key lemma]{Key Lemma} and the \hyperref[main theorem]{Main Theorem}.

We will use the following lemma.
\begin{lem}\label{l:freeness}
Let~$f$ be a rational map of degree at least~$2$ and let~$(z_n)_{n = 0}^{+ \infty}$ be a sequence in~$\julia$ such that for each integer~$n \ge 0$ we have~$f(z_{n + 1}) = z_{n}$.
Let~$\rho > 0$, $M \ge 1$ an integer and $(n_{\ell})_{\ell = 1}^{+ \infty}$ a sequence of strictly positive integers such that for each integer~$\ell \ge 1$ we have~$n_{\ell + 1} \ge n_{\ell} + M$ and the following properties hold.
There is a point~$x_{\ell}$ of~$B(z_0, \rho/2)$ in~$f^{-M}(z_{n_{\ell}})$ different from~$z_{n_{\ell} + M}$ and an inverse branch~$\phi_\ell$ of~$f^{n_{\ell} + M}$ defined on~$B(z_0, \rho)$ and taking images in~$B(z_0, \rho/2)$, such that~$\phi_{\ell}(z_0) = x_{\ell}$.
Then the IFS~$(\phi_{\ell})_{\ell = 1}^{+ \infty}$ so defined is free.
\end{lem}
\begin{proof}
Note first that the hypotheses imply that for each integer~$n \ge 1$ there is an inverse branch~$\tphi_n$ of~$f^n$ defined on~$B(z_0, \rho)$ and such that~$\tphi_n(z_0) = z_n$.

Let $k, k' \ge 1$ be integers and let $\underline{\ell} \= \ell_1 \cdots \ell_k$ and~$\underline{\ell}' \= \ell_{1}^{'} \cdots \ell_{k'}^{'}$ be different words in~$\Sigma^*$.
To prove that the maps~$\phi_{\underline{\ell}}$ and~$\phi_{\underline{\ell}'}$ are different, we assume without loss of generality that~$\ell_{k'}^{'} \ge \ell_k + 1$.
Putting~$N \= n_{\ell_1} + \cdots + n_{\ell_{k}} + kM$, we have
$$f^{N - n_{\ell_k} - M} \circ \phi_{\underline{\ell}}(z_0)
=
\phi_{\ell_k}(z_0)
=
x_{\ell_k}.$$
On the other hand, the hypothesis that for every $\ell\geq 1$ we have $n_{\ell+1}-n_{\ell}\geq M$ implies
$$ n_{\ell_{k'}^{'}} - n_{\ell_k}
\ge
n_{\ell_k + 1} - n_{\ell_k}
\ge
M $$
and hence
\begin{equation*}
f^{N-n_{\ell_k}-M}\circ \phi_{\underline{\ell}'}(z_0)
=
f^{n_{\ell_{k'}'} - n_{\ell_k}} \circ \phi_{\ell_{k'}'}(z_0)
=
\tphi_{n_{\ell_k} + M}(z_0)
=
z_{n_{\ell_k} + M}.
\end{equation*}
As by hypothesis the points~$x_{\ell_k}$ and~$z_{n_{\ell_k} + M}$ are different, it follows that the points~$\phi_{\underline{\ell}'}(z_0)$ and~$\phi_{\underline{\ell}}(z_0)$, and hence the maps~$\phi_{\underline{\ell}}$ and~$\phi_{\underline{\ell}'}$, are different.
\end{proof}

\begin{proof}[Proof of Proposition~\ref{p:construction of IFS}]
We denote by~$Z$ the space of all sequences $(z_n)_{n = 0}^{+ \infty}$ in~$\julia$ such that for each integer~$n \ge 0$ we have~$f(z_{n + 1})=z_{n}$.
We define the inverse limit of~$f : \julia \to \julia$ as the invertible map~$F : Z \to Z$ that maps a point~$(z_n)_{n = 0}^{+ \infty}$ in~$Z$ to the sequence~$(z_n')_{n = 0}^{+ \infty}$ defined by~$z_0' = f(z_0)$ and for each integer~$n \ge 1$ by~$z_n' = z_{n - 1}$.
We denote by $\Pi : Z \rightarrow \julia $ the projection given by $\Pi \left( (z_n)_{n \in \Z} \right) = z_0$, so that $f \circ \Pi = \Pi\circ F$.

Let~$\nu$ be the unique measure on~$Z$ that is invariant by~$F$ and so that $\Pi_*\nu=~\mu$.
Since~$\mu$ is ergodic, the measure~$\nu$ is also ergodic for~$F$ and $F^{-1}$.
So by~\cite[Theorem~$10.2.3$]{PrzUrb10} there is a subset of~$Z$ of full measure with respect to~$\nu$, such that for each point~$\underline{z} = (z_n)_{n = 0}^{+ \infty}$ in this set there is $r(\underline{z})>0$ such that for each $n\geq 1$ there is an inverse branch~$\tphi_n$ of~$f^n$ defined on~$B(z_0,r(\underline{z})) = B(\Pi(\underline{z}),r(\underline{z}))$ and such that~$\tphi_n(z_0) = z_{n}=\Pi(F^{-n}(\underline{z}))$.
Fix such $\underline{z}=(z_n)_{n = 0}^{+\infty}$, such that in addition it is generic for~$\nu$ and~$F^{-1}$ in the sense of the Birkhoff Ergodic Theorem, so that we have
$$ \lim_{n \to + \infty} \frac{1}{n} \sum_{j = 0}^{n - 1} \delta_{z_j} = \mu $$
in the weak* topology.
We also assume that it satisfies the conclusion of Birkhoff Ergodic Theorem for~$F^{-1}$ and for the function $\ln\vert f'\vert\circ\Pi$:
\begin{equation}\label{eq:Birko}
\lim_{n\rightarrow + \infty} \frac{1}{n}\sum_{j=0}^{n-1}\ln \vert f'\vert\circ \Pi(F^{-j}(\underline{z}))
=
\chi_\mu(f) ;
\end{equation}
and that
\begin{equation*}
\limsup_{m\rightarrow + \infty}\sum_{j=0}^{m-1}\Big(\psi\circ\Pi(F^{-j}((z_n)_{n=0}^{+ \infty}))-\int\psi d\mu\Big)\geq 0,
\end{equation*}
see for example~\cite[Lemma~$8.3$]{PrzRiv0806v2}.
Reducing~$r(\underline{z})$ if necessary, we assume that~$r(\underline{z}) < \rho_0$.
So by Koebe Distortion Theorem there are constants~$C_0 > 0$ and $\lambda_0 > 1$ such that for every~$\zeta \in B(z_0, r(\underline{z})/2)$ and every integer~$n \ge 1$ we have
\begin{equation}
\label{e:hyperbolic time}
|(f^n)'(\tphi_n(\zeta))|
\ge
C_0 \lambda_0^n.
\end{equation}
In particular, the diameter of~$\tphi_n (B(z_0, r(\underline{z})/2))$ converges to~$0$ as~$n \to +\infty$.
On the other hand there is a strictly increasing sequence of positive integers~$(n_\ell)_{\ell = 1}^{+ \infty}$ such that for all~$\ell$ we have 
$$ S_{n_{\ell}}(\psi)(z_{n_{\ell}})
\ge
n_{\ell}\int\psi d\mu - 1 $$
and such that the sequence $((z_{n_{\ell} + j})_{j = 0}^{+ \infty})_{\ell = 1}^{+ \infty}$ converges in~$Z$ to some point~$(y_n)_{n = 0}^{+ \infty}$.

Now the proof is divided into two cases, according to whether the limit point~$(y_n)_{n = 0}^{+ \infty}$ is a periodic orbit of~$F$ of exceptional points of~$f$ or not.

\partn{Case~$1$} $(y_n)_{n = 0}^{+ \infty}$ is not a periodic orbit of~$F$ of exceptional points of~$f$.
We start defining an integer~$N \ge 0$ as follows.
If~$y_0$ is a non-exceptional point we put~$N = 0$.
If~$y_0$ is an exceptional point, then our hypothesis implies that~$(y_n)_{n = 0}^{+ \infty}$ is not a periodic orbit; so there is an integer~$N \ge 1$ such that~$y_N$ is not exceptional.
In all the cases~$y_N$ is a non-exceptional point.
So there an integer~$M' \ge 1$ and an unramified preimage~$w_0'$ of~$y_N$ by~$f^{M'}$ that is not in the forward orbit of a critical point of~$f$.
By topological exactness of~$f$ on~$\julia$ there is an integer~$M \ge M' + 1$ such that there are different points~$w_0$ and~$w_1$ of~$B(z_0, r(\underline{z})/4)$ in~$f^{-(M - M')}(w_0')$.
It thus follows that
$$ f^M(w_0) = f^M(w_1) = y_N $$
and that~$f^M$ is locally injective at both, $z = w_0$ and~$z = w_1$.
Then there is~$\hrho > 0$ and an inverse branch~$\Psi_0$ (resp.~$\Psi_1$) of~$f^M$ defined on~$B(y_N, \hrho)$, taking images in~$B(z_0, r(\underline{z})/4)$ and such that~$\Psi_0(y_N) = w_0$ (resp.~$\Psi_1(y_N) = w_1$).
Replacing~$(n_{\ell})_{\ell = 1}^{+ \infty}$ by a subsequence if necessary we assume that for every integer~$\ell \ge 1$ we have
\begin{equation}
\label{e:early expansion}
\left( \sup_{B(y_N, \hrho/2)} |\Psi_0'| \right)^{-1} C_0 \lambda_0^{n_\ell + N}
\ge
\lambda_0^{(n_\ell + N + M)/2},
\end{equation}
$n_{\ell + 1} \ge n_{\ell} + M$ and
$$ \tphi_{n_{\ell} + N}(B(z_0, r(\underline{z})/2))
\subset
B(y_N, \hrho/2); $$
in particular both compositions, $\Psi_0 \circ \tphi_{n_{\ell} + N}$ and~$\Psi_1 \circ \tphi_{n_{\ell} + N}$ are defined on~$B(z_0, r(\underline{z})/2)$ and take images in~$B(z_0, r(\underline{z})/4)$.
By construction, for each integer~$\ell \ge 1$ the points~$\Psi_0 \circ \tphi_{n_{\ell} + N}(z_0)$ and~$\Psi_1 \circ \tphi_{n_{\ell} + N}(z_0)$ are different.
Replacing~$(n_{\ell})_{\ell = 1}^{+ \infty}$ by a subsequence and interchanging~$w_0$ and~$w_1$ if necessary, we assume that for each integer~$\ell$ the point
$$ x_\ell
\=
\Psi_0 \circ \tphi_{n_{\ell} + N}(z_0) = \Psi_0(z_{n_{\ell} + N}) \in f^{-M}(z_{n_\ell + N}) $$
is different from~$z_{n_{\ell} + N + M}$ and put
$$ \phi_{\ell} \= \Psi_0 \circ \tphi_{n_{\ell} + N}|_{B(z_0, \rho)}. $$
Note that~$(\phi_\ell)_{\ell = 1}^{+ \infty}$ is an IFS defined on~$B(z_0, r(\underline{z}))$ that is generated by~$f$ and whose time sequence is~$(m_{\ell})_{\ell = 1}^{+ \infty} \= (n_{\ell} + N + M)_{\ell = 1}^{+ \infty}$.

Lemma~\ref{l:freeness} with~$\rho = r(\underline{z})/2$ and with $(n_{\ell})_{\ell = 1}^{+ \infty}$ replaced by~$(n_{\ell} + N)_{\ell = 1}^{+ \infty}$, implies that the IFS~$(\phi_\ell)_{\ell = 1}^{+ \infty}$ is free.
To prove that~$(\phi_{\ell})_{\ell = 1}^{+ \infty}$ is hyperbolic with respect to~$f$, note that~\eqref{e:early expansion} implies that for each integer~$\ell \ge 1$ and each~$\zeta \in B(z_0, \rho)$
$$ |(f^{m_\ell})'(\phi_{\ell}(\zeta))|
\ge
\lambda_0^{m_\ell/2}. $$
Together with~\eqref{e:hyperbolic time} this implies that the IFS~$(\phi_\ell)_{\ell = 1}^{+ \infty}$ is hyperbolic with respect to~$f$; we omit the standard details.
It remains to prove~\eqref{e:key estimate}.
To do this, put
$$ C_1
\=
- \inf_{\julia} \psi
=
- \inf_{\julia} \varphi + t \sup_{\julia} \ln |f'|
<
+ \infty, $$
so that
\begin{multline*}
S_{m_\ell}(\psi)(\phi_{\ell}(z_0))
=
S_{N + M}(\psi)(\phi_\ell(z_0)) + S_{n_{\ell}}(\psi)(\tphi_{n_{\ell}}(z_0))
\\ \ge
 - (N + M) C_1 + n_{\ell} \int \psi d\mu - 1.
\\ =
m_{\ell} \int \psi d\mu - 1 - (N + M)\left( C_1 + \int \psi d\mu \right).
\end{multline*}
This proves~\eqref{e:key estimate} with~$C = 1 + (N + M) \left( C_1 + \int \psi d\mu \right)$ and completes the proof of the proposition in this case.

\partn{Case~$2$}
$(y_n)_{n = 0}^{+ \infty}$ is a periodic orbit of~$F$ of exceptional points of~$f$.
Since the point~$\underline{z}$ is generic for~$\nu$ and~$F^{-1}$ and since by hypothesis~$\mu$ is not supported on an exceptional periodic orbit, it follows that~$\underline{z}$ is not a periodic orbit of~$F$ of exceptional points of~$f$; in particular it is different from~$(y_n)_{n = 1}^{\infty}$.
Together with the fact that for each integer~$n \ge 1$ there is an inverse branch of~$f^n$ defined on~$B(z_0, r(\underline{z}))$ and mapping~$z_0$ to~$z_n$, namely~$\tphi_n$, this implies that~$B(z_0, r(\underline{z}))$ does not contain any exceptional point.
In particular, $\tphi_n(B(z_0, r(\underline{z})))$ does not contain~$y_0$.

By the topological exactness of~$f$ on~$\julia$ there is an integer~$M \ge 1$ such that there exists a point~$w_0$ of~$B(z_0, r(\underline{z})/8)$ in~$f^{-M}(y_0)$.
Since~$y_0$ is by hypothesis an exceptional point, $z = w_0$ is a critical point of~$f^M$.
Let~$\hrho > 0$ be sufficiently small so that the connected component~$W$ of~$f^{-M}(B(y_0, \hrho))$ containing~$w_0$ is contained in~$B(z_0, r(\underline{z})/8)$ and such that the only critical point of~$f^M$ in~$W$ is $z = w_0$.
Replacing~$(n_\ell)_{\ell = 1}^{\infty}$ by a subsequence if necessary we assume that for every integer~$\ell \ge 1$ we have~$n_{\ell + 1} \ge n_{\ell} + M$ and
$$ \tphi_{n_\ell}(B(z_0, r(\underline{z})/2))
\subset
B(y_0, \hrho). $$
Since by the above~$z_{n_\ell} \neq y_0$ and since~$f^M : W \to B(y_0, \hrho)$ is of degree at least~$2$, it follows that there are at least two points of~$W$ in~$f^{-M}(z_{n_{\ell}}) = f^{-M}(\tphi_{n_{\ell}}(z_0))$; let~$x_{\ell}$ be one of these points that is different from~$z_{n_\ell + M}$.
On the other hand, the fact that~$\tphi_{n_\ell}(B(z_0, r(\underline{z})/2))$ is contained in~$B(y_0, \hrho)$ and does not contain~$y_0$, implies that there is an inverse branch of~$f^{n_{\ell} + M}$ defined on~$B(z_0, r(\underline{z})/2)$, taking images in~$B(z_0, r(\underline{z})/8)$ and that maps~$z_0$ to~$x_{\ell}$; we denote by~$\phi_{\ell}$ the restriction of this inverse branch to~$B(z_0, r(\underline{z})/4)$, so by construction~$\phi_{\ell}$ has a univalent extension to~$B(z_0, r(\underline{z})/2)$.
Note that~$(\phi_{\ell})_{\ell = 1}^{+ \infty}$ is an IFS defined on~$B(z_0, r(\underline{z})/4)$, that is generated by~$f$ and has~$(m_{\ell})_{\ell = 1}^{+ \infty} \= (n_{\ell} + M)_{\ell = 1}^{+ \infty}$ as time sequence.

Lemma~\ref{l:freeness} with~$\rho = r(\underline{z})/4$ implies that the IFS $(\phi_\ell)_{\ell = 1}^{+ \infty}$ is free.
The proof of~\eqref{e:key estimate} is similar as in Case~$1$.
It remains to prove that, after replacing~$(n_{\ell})_{\ell \ge 1}^{+ \infty}$ by a subsequence if necessary, the IFS~$(\phi_{\ell})_{\ell = 1}^{+ \infty}$ is hyperbolic with respect to~$f$.
To do this, let~$D \ge 2$ be the local degree of~$f^M$ at~$w_0$, let~$p \ge 1$ be the minimal period of the exceptional point~$y_0$ and put~$L \= \sup_{\CC} |(f^p)'|$.
Note that we can reduce~$\hrho$ by replacing~$(n_{\ell})_{\ell = 1}^{+ \infty}$ by a subsequence; by doing this if necessary, we assume that for each~$j \in \{ 1, \ldots, p - 1 \}$ the set~$f^j(B(y_0, \hrho))$ is disjoint from~$B(y_0, \hrho)$, that for each~$n \in \{0, \ldots, p - 1 \}$ the point~$z_n$ is not contained in~$B(y_0, \hrho)$, that
\begin{equation}
\label{e:mass condition}
L^{2\mu(B(y_0, \hrho)) (D - 1)/D} < \lambda_0^{1/3}
\end{equation}
and that there is a constant~$C_2 > 0$ such that for each point~$z$ in~$B(y_0, \hrho)$ and each point~$x$ of~$W$ in~$f^{-M}(z)$ we have
$$ |(f^M)'(x)|
\ge
C_2 \dist(z, y_0)^{(D - 1)/D}. $$ 
Given an integer~$\ell \ge 1$, let~$T_{\ell}$ be the largest integer~$T \ge 1$ such that for all~$j \in \{0, \ldots, T - 1\}$ we have~$f^{jp}(z_{n_{\ell}}) \in B(y_0, \hrho)$; by our choice of~$\hrho$ we have~$p T_{\ell} \le n_{\ell}$.
Since~$\lim_{n \to + \infty} \frac{1}{n} \sum_{n = 1}^{+ \infty} \delta_{z_n} = \mu$, for every sufficiently large integer~$\ell \ge 1$ we have, 
\begin{equation}
\label{e:frequency condition}
T_{\ell}
\le
2 \mu(B(z_0, \rho)) n_{\ell} .
\end{equation}
Replacing~$(n_{\ell})_{\ell = 1}^{+ \infty}$ by a subsequence if necessary, we assume this holds for every integer~$\ell \ge 1$.
So, by the definition of~$L$, for every integer~$\ell \ge 1$ we have
$$ \dist(z_{n_\ell}, y_0)
\ge
L^{- T_{\ell}} \hrho. $$
Letting~$C_3 \= C_2 \hrho^{(D - 1)/D}$ we have by~\eqref{e:frequency condition} and then by~\eqref{e:mass condition},
\begin{multline*}
|(f^M)'(x_{\ell})|
\ge
C_2 \dist(z_{n_\ell}, y_0)^{(D - 1)/D}
\\ \ge
C_3 L^{- T_{\ell}(D - 1)/D}
\ge
C_3 \left( L^{2\mu(B(y_0, \hrho))(D - 1)/D} \right)^{- n_{\ell}}
\ge
C_3 \lambda_0^{- n_{\ell}/3}.
\end{multline*}
We thus have by the definition of~$C_0$ and~$\lambda_0$,
$$ |(f^{m_{\ell}})'(\phi_{\ell}(z_0))|
=
|(f^{n_{\ell}})'(z_{n_{\ell}})| \cdot |(f^M)'(x_{\ell})|
\ge
C_0 C_3 \lambda_0^{2 n_{\ell}/3}. $$
Since~$\phi_\ell$ has a univalent extension to~$B(z_0, r(\underline{z})/2)$, by Koebe Distortion Theorem there is a constant~$C_4 > 0$ independent of~$\ell$ so that for each~$\zeta \in B(z_0, r(\underline{z})/4)$ we have
$$ |(f^{m_{\ell}})'(\phi_{\ell}(\zeta))|
\ge
C_4 \lambda_0^{2 n_{\ell}/3}. $$
So, replacing~$(n_{\ell})_{\ell = 1}^{+ \infty}$ by a subsequence if necessary, we have for every integer~$\ell \ge 1$ and every~$\zeta \in B(z_0, r(\underline{z})/4)$,
$$ |(f^{m_{\ell}})'(\phi_{\ell}(\zeta))|
\ge
\lambda_0^{n_{\ell}/3}. $$
Together with~\eqref{e:hyperbolic time} this implies that the IFS $(\phi_{\ell})_{\ell = 1}^{+ \infty}$ is hyperbolic with respect to~$f$ and finishes the proof of the proposition.
\end{proof}


\end{document}